\documentclass{gtpart}   
\usepackage{graphicx,subfig}
\usepackage{enumerate} 
\usepackage[utf8]{inputenc} 
\usepackage[T1]{fontenc}
\usepackage{ae,aecompl}
\usepackage{dsfont} 
\usepackage{pxfonts}
\usepackage{microtype}
\usepackage{braket}

\usepackage{yhmath}

\usepackage{rotating}

\newif\ifPDF
\ifx\pdfoutput\undefined 
        \ifx\pdfversion\undefined
                \PDFfalse
        \fi
\else 
        \ifnum\pdfoutput > 0
                \PDFtrue
        \else
                \PDFfalse
        \fi
\fi

\makeatletter
\newcommand{\xRightarrow}[2][]{\ext@arrow 0359\Rightarrowfill@{#1}{#2}}
\makeatother

\theoremstyle{plain}
				
\newtheorem{theorem}{Theorem}

\newtheorem{lemma}[theorem]{Lemma}

\theoremstyle{definition}

\newtheorem{remark}[theorem]{Remark}

\newcommand{\CBbb}{\mathbb C}

\newcommand{\HBbb}{\mathbb H}

\newcommand{\NBbb}{\mathbb N}

\newcommand{\RBbb}{\mathbb R}

\newcommand{\Fcal}{\mathcal F}

\newcommand{\Mcal}{\mathcal M}

\newcommand{\Xcal}{\mathcal X}

\newcommand{\tec}{Teichm\"uller }
\newcommand{\SL}{\mathsf{SL}}
\newcommand{\PSL}{\mathsf{PSL}}

\newcommand{\SU}{\mathsf{SU}}

\DeclareMathOperator{\Hom}{Hom}

\DeclareMathOperator{\Teich}{Teich}
\DeclareMathOperator{\dev}{dev}

\newcommand{\lra}{\longrightarrow}
\newcommand{\doubleslash}{\bigr/ \negthinspace\negthinspace \bigr/}

\newcommand{\Iso}{\mathsf{Iso}}

\newcommand{\isorightarrow}{\xrightarrow{
   \,\smash{\raisebox{-0.5ex}{\ensuremath{\sim}}}\,}}

\begin{document}


\title{Surface group  representations to $\SL(2,\CBbb)$ and Higgs bundles with smooth spectral data}

\author{Richard A. Wentworth}

\address{Department of Mathematics,
   University of Maryland,
   College Park, MD 20742, USA}
\email{raw@umd.edu}

\author{Michael Wolf}
\address{Department of Mathematics, Rice University,
Houston, TX 77251}
\email{mwolf@rice.edu}



\begin{abstract} 
We show that for every nonelementary representation of a surface group into $\SL(2,\CBbb)$ there is a Riemann surface structure such that the  Higgs bundle associated to the representation  lies outside the discriminant locus of the Hitchin fibration. 
\end{abstract}



\maketitle
\setcounter{tocdepth}{2}
\thispagestyle{empty}





\section{Introduction}
Let $\Sigma$ be a closed, oriented surface of genus $g\geq 2$. 
In this short note we answer a special case of the following question posed by 
Nigel Hitchin: which representations $\rho: \pi_1(\Sigma)\to \SL(n,\CBbb)$ correspond to Higgs bundles which lie outside the discriminant locus of the Hitchin fibration for some Riemann surface structure on $ \Sigma$?  For example, the Higgs field for a unitary representation (i.e.\ one whose image lies in a conjugate of $\SU(n)$) is identically zero, and a reducible representation (i.e.\ one whose image preserves a proper subspace of $\CBbb^n$ for the standard action) necessarily has a Higgs field whose characteristic polynomial is reducible. As a consequence, these representations always lie in fibers over the discriminant locus for any choice of Riemann surface structure. 
The goal of this paper is to show that for $n=2$ these examples present essentially  the only restrictions.  To state the result, recall that a  representation $\rho:\pi(\Sigma)\to \SL(2,\CBbb)$ is called \emph{elementary} if it is either   unitary, reducible, or maps to the subgroup generated by an embedding
 $$
 \CBbb^\ast\hookrightarrow \SL(2,\CBbb)\ :\ \lambda\mapsto \left(\begin{matrix} \lambda&0\\0&\lambda^{-1}\end{matrix}\right)
 $$
 and the element $\left(\begin{matrix} 0&-1\\ 1&0\end{matrix}\right)$.
  We shall prove the following

\begin{theorem} \label{thm:main}
A  semisimple representation $\rho:\pi_1(\Sigma)\to \SL(2,\CBbb)$ defines a point in the fiber of the Hitchin fibration over the discriminant locus for every Riemann surface structure on $\Sigma$ if and only if $\rho$ is elementary.
\end{theorem}

The natural approach to the above statement is to prove that if  $\rho$ is nonelementary, one can find a Riemann surface structure $X$ on $\Sigma$ so that the Higgs bundle on $X$ corresponding to $\rho$ defines a point in the fiber of the Hitchin fibration away from the discriminant locus for $X$.  We shall prove this  by combining the powerful result of Gallo-Kapovich-Marden \cite{gallo-kapovich-marden} with the method of harmonic maps to trees \cite{wolf:trees}, \cite{wolf:mfs-harmonic}.  

Let us first review a bit of the background and terminology for this problem. Let
\begin{equation} \label{eqn:character}
\Xcal(\Sigma)=\Hom(\pi_1(\Sigma), \SL(2,\CBbb))\doubleslash\SL(2,\CBbb)
\end{equation}
denote the $\SL(2,\CBbb)$-character variety of $\Sigma$ parametrizing semisimple representations (see \cite{CullerShalen:1983, LubotzkyMagid:1985}).  For a (marked) Riemann surface structure $X$ on $\Sigma$, let $\Mcal(X)$ denote the moduli space of rank $2$ Higgs bundles on $X$ with fixed trivial determinant (see \cite{Hitchin:87}). The nonabelian Hodge theorem  asserts the existence of a homeomorphism $\Xcal(\Sigma)\simeq \Mcal(X)$ for each $X$.  One direction of the  homeomorphism is a consequence of the following result of Corlette and Donaldson \cite{Corlette:88, Donaldson:87}: given a semisimple representation $\rho:\pi_1(\Sigma)\to \SL(2,\CBbb)$
and a Fuchsian representation $\sigma: \pi_1(\Sigma)\isorightarrow\Gamma\subset \PSL(2,\RBbb)$, $X=\Gamma\backslash \HBbb^2$, 
 there exists a smooth harmonic  map $v:\HBbb^2\to \HBbb^3$ that is equivariant for the action of $\pi_1(\Sigma)$ via $\sigma$ on the upper half plane $\HBbb^2\subset\CBbb$ and $\rho$ on the hyperbolic $3$-space
  $\HBbb^3$,  on which $\SL(2,\CBbb)$ acts by isometries. Moreover,  $v$ minimizes the energy among all such equivariant maps.
  We shall refer to $v$ as an \emph{equivariant harmonic map}. 
  If $Q(X)$ denotes the space of holomorphic quadratic differentials on $X$, then there is a (singular) holomorphic fibration $h:\Mcal(X)\to Q(X)$ which is a smooth fibration of abelian varieties over the locus of nonzero differentials with simple zeros. 
   The image by $h$ of a Higgs bundle corresponding to  a semisimple representation  is simply the Hopf differential of any equivariant harmonic map, as described  above. 
  The divisor  $\Delta(X)\subset Q(X)$ consisting  of those quadratic differentials having some zero with multiplicity is called the \emph{discriminant locus}.
  Points in $\Mcal(X)$ in  the fiber over  $q\in Q(X)\setminus\Delta(X)$  correspond to certain line bundles on a branched double cover of $X$ called the \emph{spectral curve}. The line bundle and the spectral curve together  form the \emph{spectral data}, which completely determine the Higgs bundle, and hence via the other direction of the nonabelian Hodge theorem,  the corresponding representation $\rho$. The spectral data for points in $\Mcal(X)$ lying over the discriminant locus are more difficult to describe; hence, the interest in the question posed by 
  Hitchin. For more on this structure, see \cite{Hitchin:1987b}.
  
   With this understood,  Theorem \ref{thm:main} is a direct consequence of the following equivalent statement.
 
 \begin{theorem} \label{thm:harmonic}
 Let $\rho:\pi_1(\Sigma)\to \SL(2,\CBbb)$ be a semisimple representation. Then there exists a Riemann surface structure $X=\Gamma\backslash\HBbb^2$ on $\Sigma$ such that the Hopf differential of the $\rho$-equivariant harmonic map $\HBbb^2\to \HBbb^3$ has only simple zeros
 if and only if $\rho$ is nonelementary.
 \end{theorem}
 
 \begin{remark}
 \begin{enumerate}
 \item 
A unitary representation fixes a point in $\HBbb^3$, and so  the constant map is equivariant and clearly energy minimizing.
Hence,  the Hopf differential vanishes. A semisimple elementary representation that is not unitary fixes a geodesic in $\HBbb^3$, which  then necessarily coincides with the  image of any equivariant harmonic map.   The Hopf differential is therefore the square of an abelian differential. In particular, since we assume $g\geq 2$, the differential has zeros with multiplicity.
 Therefore, the ``only if'' parts of Theorems \ref{thm:main} and \ref{thm:harmonic} are clear. 
 \item We shall actually prove a slightly stronger statement; namely, for nonelementary representations we can find a Riemann surface structure such that the vertical foliation of the Hopf differential has no saddle connections.
 \item 
Note that there are obviously sections of the bundle of holomorphic quadratic differentials  over Teichm\"uller space $\Teich(\Sigma)$  which at every point  have  zeros with multiplicity; one class of examples are the squares of abelian differentials just mentioned. Hence, Theorem \ref{thm:harmonic} does not seem to follow from a  simple  dimension count. 
\item As pointed out by Hitchin, there will be other obstructions in any generalization of Theorem \ref{thm:main} for $n\geq 3$. In particular, some of these will come from  other real forms of $\SL(n,\CBbb)$.  Representations to $\SU(p,q)$, $p\neq q$, for example, will always lie in the discriminant locus (cf. \cite{Schaposnik:Thesis}).
 Finding a suitable replacement  in higher rank for the result of Gallo-Kapovich-Marden  remains a challenge.
 
 \item Theorem~2 states that if a nonelementary representation $\rho$ is in the discriminant locus relative to one Riemann surface $X$, there is another Riemann surface $Y$ for which $\rho$ is not in the discriminant locus relative to $Y$. Neither the statement nor the proof suggests any conclusion about the frontier of the closure of these discriminant loci in any of the natural compactifications of the moduli space $\Xcal(\Sigma)$ (cf.\ \cite{DaskalDostoglouWentworth:00}).

\end{enumerate}
\end{remark}

We briefly outline the strategy for proving Theorem~2. 
The \emph{Hopf differential} $\Phi$ of a locally defined harmonic map $w: \Omega \to (N,d)$ from a domain $\Omega \subset \CBbb$ to metric space $(N,d)$ is defined to be $\Phi = 4 (u^* d_T)^{2,0}$. In the case of an equivariant harmonic map $v:\HBbb^2\to \HBbb^3$, this differential $\Phi$ descends to a holomorphic quadratic differential on the Riemann surface $X = \Gamma\backslash\HBbb^2$, 
which for convenience we continue to denote by $\Phi$.
  Projecting in $\HBbb^2$ along the leaf space of $\Phi$ yields a $\pi_1(\Sigma)$-equivariant harmonic map from 
$\HBbb^2$ to a metric tree $T_{\Phi}$, called the \emph{dual tree to $\Phi$}.  There are two relevant observations: first, if the vertices of $T_{\Phi}$ all have  valence three, then we recognize that $\Phi$ must have had only simple zeroes.  Second, we notice that the $\pi_1(\Sigma)$-equivariant product map $\HBbb^2 \to \HBbb^3 \times T_{\Phi}$ is  a \emph{conformal harmonic map}, i.e.\ it is both harmonic and has vanishing Hopf differential.

The idea then is to reverse this construction: we seek a tree $T$ all of whose vertices are trivalent and a $\pi_1(\Sigma)$-equivariant 
 conformal harmonic map  $\HBbb^2\to\HBbb^3 \times T$. In that case, the resulting Hopf differential for the harmonic map to $\HBbb^3$ will be the negative of the Hopf differential for projection to the tree, which necessarily has only simple zeroes.

If the function on \tec space which records, for each domain Riemann surface, the equivariant energy of the harmonic map to $\HBbb^3 \times T$ is proper, then there exists a conformal harmonic map.
Now given a representation $\rho: \pi_1(\Sigma)\to \SL(2,\CBbb)$,  then unless $\rho$ is quasi-Fuchsian one expects there to be certain divergent sequences in \tec space along which the energy of the equivariant harmonic map $\HBbb^2\to \HBbb^3$ is uniformly bounded, while for other divergent sequences the energy tends to $+\infty$. A similar statement holds for harmonic maps to trees.
 Therefore, the challenge is  to associate a tree $T$ to  a given $\rho$ such that the sum of the  energies to $T$ and $\HBbb^3$ diverges along every choice of proper path in \tec space. We are rescued in this quest by the main result of \cite{gallo-kapovich-marden}, which by realizing $\SL(2,\CBbb)$-representations of surface groups as holonomies of complex projective structures, provides a measured lamination on the surface with image of at least moderate length (quotient) length in $\HBbb^3$: that measured lamination can be adjusted so that its dual tree suffices for our needs. 

\medskip
\textbf{Acknowledgements.} 
The authors warmly thank Nigel Hitchin for introducing us to the problem, and Shinpei Baba for discussions about projective structures which led to the formulation of Lemma \ref{lem: pleating locus}. We are also appreciate the helpful comments of the referees.
 The authors were supported by National Science Foundation grants DMS-1406513 and DMS-1007383, respectively. Both authors are grateful to the Mathematical Sciences Research Institute at Berkeley, where some of this research was conducted, to the Institute for Mathematical Sciences at the National University of Singapore where the work was begun and to NSF grants DMS-1107452, 1107263, 1107367 ``RNMS: GEometric structures And Representation varieties"(the GEAR Network) which supported collaborative travel for the authors. \\

\section{Trees, measured foliations,  and harmonic maps} \label{sec:trees}
In this section, we prove a lemma that motivates the strategy of the proof of Theorem~\ref{thm:harmonic}. The basic constructions in the statement of the lemma below were first exploited in \cite{wolf:mfs-harmonic}.
Namely, we will find the desired Riemann surface structure   as a critical point for an energy function on \tec space. 
To define this energy function, first choose a measured foliation, say $(\Fcal, \lambda)$ on the differentiable surface $\Sigma$, lift that measured foliation to a $\pi_1(\Sigma)$-equivariant measured foliation on the universal cover 
$\widetilde{\Sigma}$, and then project the transverse measure $\lambda$ along the leaves to obtain an $\RBbb$-tree $T=T_{\lambda}$ with an isometric action (relative to the metric defined by the projected measure) of $\pi_1(\Sigma)$. 
For concreteness, we will express  the isometric action of the fundamental group on $T$ by a representation $\rho_T: \pi_1(\Sigma)\to {\mathsf{Iso}}(T)$.
 For any $\gamma\in \pi_1(\Sigma)$ 
whose free homotopy class is represented by a simple closed curve, the intersection $i(\gamma, \lambda)$ with the foliation is equal to the translation length $\gamma$ as it acts  on $T$:
\begin{equation} \label{eqn:intersection-translation}
i(\gamma, \lambda)=\vert\rho_T(\gamma)\vert_T:=\min_{x\in T} d_T(x,\gamma x)\ .
\end{equation}
Recall that the  action of an isometry on an $\RBbb$-tree is  always semisimple (cf.\ \cite{CullerMorgan:1987}); hence the ``min'' instead of an ``inf''  in \eqref{eqn:intersection-translation}.

We focus initially on two features of this construction.
First, given 
a Fuchsian representation $\sigma: \pi_1(\Sigma)\isorightarrow\Gamma\subset \PSL(2,\RBbb)$, a Riemann surface $X=\Gamma\backslash\HBbb^2$, 
and an $\RBbb$-tree $T$ with an isometric action $\rho_T:\pi_1(\Sigma)\to \Iso(T)$,
a map $u:\HBbb^2\to T$ is called \emph{$\pi_1(\Sigma)$-equivariant} if $u(\sigma(\gamma) z)=\rho_T(\gamma)u(z)$ for all $\gamma\in \pi_1(\Sigma)$ and all $z\in \HBbb^2$ (when the Riemann surface structure is assumed, we sometimes say that $u$ is \emph{$\rho_T$-equivariant} to emphasize the action on the target).
We define the $\rho_T$-energy $E_{\rho_T}(X)$ of $X$ to be the infimum of  the energies of locally finite energy $\pi_1(\Sigma)$-equivariant  maps $\HBbb^2 \to T$ (see \cite{wolf:trees} for the case of maps to $\RBbb$-trees and  \cite{KorevaarSchoen:1993, Jost:1994} for the general setting of nonpositively curved metric space targets). 
 Here we note that the energy density for such maps is a locally integrable form on $\HBbb^2$ that is invariant with respect to the action of $\pi_1(\Sigma)$ via $\sigma$.
It therefore descends to $X$, and its integral gives a well defined (finite) energy.
  Moreover, any energy minimizer (or \emph{harmonic map}) $u:\HBbb^2\to T$ has the following property:
  \begin{itemize}
  \item
   there is a nonzero holomorphic quadratic (Hopf) differential $\Phi\in Q(X)$
   whose vertical measured foliation (on $\HBbb^2$) defines a metric tree $T_\Phi$ with an isometric action of $\pi_1(\Sigma)$; 
   \item there is a $\pi_1(\Sigma)$-equivariant map $\psi:T_\Phi\to T$ which is a \emph{folding}; in case $T=T_\lambda$ is dual to a measured foliation (the only case we will consider here), then $\psi$ is an isometry;
   \item finally, $u=\psi\circ \pi$, where $\pi: \HBbb^2\to T_\Phi$ is the projection onto the vertical leaf space of $\Phi$;
  \end{itemize}
  (see \cite{hubbard-masur:foliations, wolf:trees, wolf:hubbard-masur, DaskalDostoglouWentworth:00}). 
  Moreover, the energy of $u$ is given by 
  \begin{equation} \label{eqn:energy}
  E_{\rho_T}(X):=E(u)=2\int_X \vert \Phi\vert\ .
  \end{equation}
  The energy only depends on the marked isomorphism class of $X$. Hence, $E_{\rho_T}(X)$ is a well-defined function  $E_{\rho_T}:\Teich(\Sigma)\to \RBbb_{\geq 0}$.

Second, some features of the (Hopf) quadratic differential $\Phi$ are reflected in the tree: in particular, if each vertex of the tree has valence three, then  $\Phi$ can have only simple zeros, as any higher order zeros  -- or indeed any collection of zeros connected by subarcs of a leaf -- would create higher order branching of the leaf space, which is the tree $T=T_\Phi$ in this setting. As it is a generic condition that the zeros of a holomorphic quadratic differential should be simple with no connecting leaves between them, we see that the generic tree dual to a measured foliation should have all vertices of valence three.

 The hyperbolic $3$-ball $\HBbb^3=\SL(2,\CBbb)/\SU(2)$ has a left action of $\SL(2,\CBbb)$ by isometries. 
 Fix a semisimple representation $\rho: \pi_1(\Sigma)\to \SL(2,\CBbb)$. Then $\HBbb^3$ inherits a left action of $\pi_1(\Sigma)$ by isometries.
 Given a Riemann surface structure $X=\Gamma\backslash\HBbb^2$ on $\Sigma$,  the theorem of Corlette-Donaldson mentioned in the introduction asserts the existence of  a harmonic map $v:\HBbb^2\to \HBbb^3$ that is equivariant with respect to $\rho$; this map is unique if and only if $\rho$ is irreducible. 
  Thus, in analogy with what we did with the target tree $T$ in defining the $\rho_T$-energy, we may define  the $\rho$-energy $E_\rho(X)$ of a Riemann surface  to be the 
energy $E(v)$ of $v$. As before, the function $E_\rho$ is well-defined on the \tec space $\Teich(\Sigma)$.

Finally, consider the nonpositively curved metric space $N=T\times \HBbb^3$ with product metric $d_N$ and the diagonal isometric action 
 $\pi_1(\Sigma)\to \Iso(N)$ given by  $\rho_N(\gamma)=(\rho_T(\gamma), \rho(\gamma))$.
 The energy of equivariant maps $\HBbb^2\to N$ is simply the sum of the energies of the maps to $T$ and $\HBbb^3$.
This defines our setting well enough to state 

\begin{lemma} \label{lem: sufficiency of energy properness}
Let $T=T_\lambda$ be a tree which is both dual to a measured foliation on the surface $\Sigma$ and has all vertices of valence three, and let $\rho:\pi_1(\Sigma)\to \SL(2,\CBbb)$ be irreducible.   Suppose that the function $E_{\rho_N}=E_{\rho_T}+E_\rho$ is proper on $\Teich(\Sigma)$.	Then there exists a Riemann surface structure on $\Sigma$ such that the Hopf differential  of the $\rho$-equivariant harmonic map $\HBbb^2 \to \HBbb^3$ has only simple zeros.
\end{lemma}

\begin{remark}
By our comments above on the generic nature of such trees, we see that the first sentence is not a vacuous condition. 	
\end{remark}

\begin{proof} 
By a classical result 
(see \cite{schoen-yau:incompressible, SacksUhlenbeck:82}, and for the  case of general nonpositively curved metric  target spaces, \cite[Corollary 1.3]{Wentworth:07}), the energy function $E_{\rho_T} + E_{\rho}: \Teich(\Sigma) \to \RBbb$ is differentiable on  $\Teich(\Sigma)$, 
and so, being proper, achieves its minimum at a point $X=\Gamma\backslash\HBbb^2$; moreover, the gradient of that energy function vanishes at $X$. On the other hand, the classical expression for the gradient as a multiple of the Hopf differential of the $\rho_N$-equivariant harmonic map from $\HBbb^2$ to $T \times \HBbb^3$ holds in this case (see \cite[Theorem 1.2]{Wentworth:07}), and so the Hopf differential of that harmonic map vanishes. Because the target metric is a product, we may express the harmonic map $f: \HBbb^2 \to T \times \HBbb^3$ as a product $f=(u, v)$, where
$u$ is the unique $\rho_T$-equivariant harmonic map $ \HBbb^2 \to T$, and $v$ is the unique $\rho$-equivariant harmonic map $ \HBbb^2 \to \HBbb^3$.  The Hopf differential of  $f$ is the sum of the Hopf differentials $\Phi_u$ and $\Phi_v$ of $u$ and $v$, respectively; and since it vanishes, we have $\Phi_v=-\Phi_u$.
However, as explained in the opening of this section, 
the vertical measured foliation of  $\Phi_u$  has leaf space which projects to a tree $T_{\Phi_u}$ that is equivariantly isometric to $T$. In particular, since $T$ has all vertices of valence three,  the differential $\Phi_u$ has simple zeros.  The same is therefore true of  $\Phi_v=-\Phi_u$.
\end{proof}

\section{Complex projective structures and bending laminations} \label{sec:projective}

Let us introduce some more notation. 
 For a hyperbolic surface $S$ and simple closed curve $\gamma\subset S$, let $\ell_S(\gamma)$ denote the length of the geodesic in the free homotopy class of $\gamma$ as measured on $S$.
For $g\in \Iso(\HBbb^3)$, define the translation length $\vert g\vert_{\HBbb^3}$ as in eq.\ \eqref{eqn:intersection-translation}:
\begin{equation*} \label{eqn:translation}
\vert g\vert_{\HBbb^3} := \inf_{x\in \HBbb^3} d_{\HBbb^3}(g\cdot x, x) \ .
\end{equation*}

The goal now is to find a tree  for which the hypotheses of Lemma~\ref{lem: sufficiency of energy properness} are satisfied.
To that end, let $\rho:\pi_1(\Sigma)\to \SL(2,\CBbb)$ be nonelementary. The foundational result in \cite{gallo-kapovich-marden} 
implies that $\rho$ is the holonomy of a complex projective structure, say $(X,\wp)$, and hence is the holonomy of a developing map $\dev_\rho:\widetilde{\Sigma} \to \CBbb P^1$. (The reader may find it useful to keep in mind that this complex projective structure is not necessarily unique, and in general, the developing map, while a local homeomorphism, is neither necessarily injective nor a covering.)
 We exploit the rich synthetic hyperbolic geometry of complex projective structures in the following lemma; in that setting, because of hyperbolic geometric constructions, it is more convenient to replace measured foliations with measured laminations in the discussion.  As the natural homeomorphism between the space of measured foliations and measured laminations respects the passage to dual trees, there is no loss of content in this change of perspective. 
 A \emph{maximal lamination} is a 
measured lamination all of whose complementary regions are ideal 
triangles; or more background on properties of geodesic laminations used below, see \cite{Bonahon:1986}.

\begin{lemma} \label{lem: pleating locus}
Let $(X, \wp)$ be a complex projective structure on $\Sigma$ with holonomy $\rho$. Then there is a hyperbolic structure $S$ on $\Sigma$, a maximal 
measured geodesic lamination $\lambda$ on $S$, 
and constants $\varepsilon_1, A>0$, depending only on $(S, \lambda)$, such that the following hold:
\begin{enumerate}
\item if $\gamma$ is a simple closed curve on $\Sigma$ with intersection number 
$i(\gamma, \lambda) < \varepsilon_1$, then $\vert\rho(\gamma)\vert_{\HBbb^3}\geq A\ell_S(\gamma)$;
\item more generally, for any constant $I > 0$, there is  $L>0$ so 
that if $\gamma$ is a simple closed  curve on $\Sigma$ with 
$i(\gamma, \lambda) < I$ and $\ell_S(\gamma) > L$, then $\vert\rho(\gamma)\vert_{\HBbb^3}\geq A\ell_S(\gamma)$.
\end{enumerate}
\end{lemma}

\begin{proof}
We begin by recalling the key property  of complex projective structures we will need.  Good references for this material, due almost entirely to Thurston, are \cite[Section 2]{kamishima-tan:deformation-geometric} and \cite[Theorem 8.6]{kulkarni-pinkall:canonical}.
Given a complex projective structure $(X, \wp)$ on $\Sigma$ with holonomy $\rho$,  
there is  a hyperbolic surface structure $S$ on $\Sigma$, a measured geodesic lamination $\lambda_0$ and a (pleated surface) map $F:\widetilde{S} \to \HBbb^3$ from the universal cover $\widetilde{S}$ to $\HBbb^3$, which 
has image a  
surface $F(\widetilde S) \subset \HBbb^3$ and for which $F\bigr|_{\tilde{\lambda_0}}$ is an isometry. Here, $\tilde{\lambda_0}$ is the lift to $\widetilde S$ of the lamination $\lambda_0 \subset S$.

Choose a point $p \in \lambda_0$ and a small neighborhood $U \subset S$ containing $p$. Some of the leaves, say $\alpha_i$, of $\lambda_0$ that meet $U$ later recur to $U$, and the images of those arcs $\alpha_i$ determine $F$-images, say $F(\widehat{U_i}) = V_i \subset \HBbb^3$, of lifts $\widehat{U_i}$ of $U$ that are separated by (fixed portions of) 
isometric images of the arcs $\alpha_i$. In particular, the images $V_i$ of those lifts are at some minimum distance $A$ from each other, depending only on the geometry of $S$ and $\lambda_0 \subset S$.

Note that if $\gamma$ is a closed curve which lies $C^1$-close to a lamination, 
we can choose such a neighborhood $U$ so that $\gamma$ meets $U$ some number $k$ times before closing up.  Thus, if 
the image $F(\tilde{q})$ of 
a lift $\tilde{q}$ of a point $q \in \gamma \cap U$ were to lie in a neighborhood $V_0 \subset \HBbb^3$, then the image $\rho(\gamma)(F(\tilde{q}))$ by the isometry $\rho(\gamma)$ of  $F(\tilde{q})$ would have to lie in some lift $V_k \subset \HBbb^3$, 
with a single lift $\hat{\gamma}$ connecting the neighborhoods $V_0$ and $V_k$ and meeting other lifts $V_1,...,V_{k-1}$ along its path. We conclude that such an isometry $\rho(\gamma)$ has translation length $\vert\rho(\gamma)\vert_{\HBbb^3}$ comparable to that of its length $\ell_S(\gamma)$ on $S$: the construction shows that this comparability constant $\vert\rho(\gamma)\vert_{\HBbb^3}/\ell_S(\gamma)$ may be taken to depend only on $\lambda_0$ and $S$, but to be independent of $\gamma$, so long as $\gamma$ is sufficiently close in $C^1$ to $\lambda_0$.

Therefore, if $\lambda_0$ is also a maximal lamination, set $\lambda = \lambda_0$ and our construction of $\lambda$ is complete.
It is of course possible that the lamination $\lambda_0$ is not maximal. For example, the lamination $\lambda_0$ might consist only of a single simple closed curve, so that the complement 
in $\Sigma$ of $\lambda_0$ could be a surface of large Euler characteristic.
In that case, we may  perturb $\lambda_0$ into a maximal lamination $\lambda$: measured 
laminations which are maximal in this sense are dense, for example by using \cite{hubbard-masur:foliations} and the density of holomorphic quadratic differentials on a Riemann surface with corresponding properties or the theory of train tracks \cite{Penner-Harer}. 
This new measured lamination $\lambda$ 
will meet the old lamination $\lambda_0$ at a maximum angle  $\delta>0$, which we may choose to be as small as we wish. In particular, the perturbation of $\lambda_0$ to $\lambda$ has only a mild effect on our constructions and estimates:  by choosing $\delta$ small enough, and restricting ourselves to curves $\gamma$ which are both very long and very close in $C^1$ to leaves in $\lambda$, we find that since $\lambda$ is close to $\lambda_0$ in $C^1$, we have already focused on curves which are sufficiently close to $\lambda_0$ in $C^1$ for the previous estimates to hold: for curve classes $\gamma$ whose $S$-geodesic representatives are sufficiently close to the $S$-measured geodesic lamination $\lambda$, we have that $\vert\rho(\gamma)\vert_{\HBbb^3}\geq A\ell_S(\gamma)$.

With these observations in mind, consider part (ii) of the lemma.
Fix a number $I>0$.
 It suffices to show that  there is a bound
  $L>0$ such that for any  simple closed 
  curve $\gamma\subset\Sigma$ with intersection number 
$i(\gamma, \lambda) < I$ and  length $\ell_S(\gamma) > L$,  the $S$-geodesic representative of $\gamma$ lies $C^1$-close to the $S$-geodesic measured lamination $\lambda$.
For suppose that it is not the case, i.e. that there is some $I$ and a sequence of curves $\gamma_k$ for which $i(\gamma_k, \lambda) < I$, while $\ell_S(\gamma_k) \to \infty$ and the $C^1$-distance between $\gamma_k$ and $\lambda$ is bounded away from zero.  
Consider the measured geodesic laminations $\mu_k$ whose measure is given, for a transverse arc $C$, by $\mu_k(C) = i(C, \gamma_k)/\ell_S(\gamma_k)$, i.e. normalized counting measure.  Of course,  as $k \to \infty$, the intersection numbers satisfy
$$i(\mu_k, \lambda) = i(\gamma_k/\ell_S(\gamma_k), \lambda) < \frac{I}{\ell_S(\gamma_k)} \to 0\ .$$  
Allowing $\mu$ to be an accumulation point of $\mu_k$, we see  that $i(\mu, \lambda) =0$. Moreover, $\mu$ is nontrivial (for example, a subsequence $\mu_k$ can all be carried on a single train track, but then one of the finitely many branches of that track admits an intersection number with a transverse arc that is bounded away from zero).  But as $\lambda$ is maximal and $i(\mu, \lambda) =0$, we have that $\mu$ is a sublamination of $\lambda$, hence the support of $\mu_k$ -- that is, the curve $\gamma_k$ -- may be taken to approximate $\lambda$ in the Hausdorff sense. This in turn implies, by the geometry of nearby hyperbolic geodesics, that $\gamma_k$ lies
arbitrarily closely to $\lambda$ in $C^1$, contradicting the assumption.

Similarly, for part (i), if no such constants $\varepsilon_1, A$ exist, we may find $\gamma_k$ for which $i(\gamma_k, \lambda)\to 0$
and $\vert\rho(\gamma_k)\vert_{\HBbb^3}/\ell_S(\gamma_k)\to 0$,
 and we derive a contradiction as above.
This completes the proof of the lemma.
\end{proof}

\section{Proof of the main result}

Let $\rho:\pi_1(\Sigma)\to\SL(2,\CBbb)$ be non-elementary. 
The theorem of Gallo-Kapovich-Marden guarantees that $\rho$ is the holonomy of a complex projective structure $(X,\wp)$ on $\Sigma$.  Let $T=T_\lambda$ be the dual tree to the measured lamination, and $S$ the hyperbolic structure on $\Sigma$, obtained  in Lemma \ref{lem: pleating locus}. Let  $N=T\times \HBbb^3$ 
and $\rho_N$ be as in Section \ref{sec:trees}.
We will need a preliminary result about $N$:
by Lemma~\ref{lem: pleating locus} (i) and eq.\ \eqref{eqn:intersection-translation}, 
we immediately have 
\begin{lemma} \label{lem:systole}
There exists $\varepsilon_2>0$, depending only on $\rho$, $S$, and $\lambda$,  such that for all $1\neq \gamma\in \pi_1(\Sigma)$, the translation length $\vert\rho_N(\gamma)\vert_N \geq \varepsilon_2$.
\end{lemma}

We can now give the
\begin{proof}[Proof of Theorem~\ref{thm:harmonic}]
By  Lemma  \ref{lem: sufficiency of energy properness}, it suffices to show  that the energy function $E_{\rho_N}=E_{\rho_T} + E_{\rho}$ is proper on $\Teich(\Sigma)$.
 Let us remark that in case $\rho$ is quasi-Fuchsian, it was shown in \cite[Section 5]{GoldmanWentworth:07}
 (see also \cite[Prop.\ 3.6]{wolf:mfs-harmonic}) that $E_\rho$ is proper, and therefore so is $E_{\rho_N}$ for any choice of  $T$.  
 For general $\rho$, however, properties of the lamination $\lambda$ and the associated tree $T=T_\lambda$ play a key role, and 
 the argument is necessarily different from the one used in \cite[Section 5]{GoldmanWentworth:07}.
 With the intent of  arriving at a contradiction, we therefore suppose to the contrary that $E_{\rho_N}$ is not proper. 
 Under the assumption  we can find a sequence  $\sigma_i : \pi_1(\Sigma)\isorightarrow \Gamma_i\subset \PSL(2,\RBbb)$  of Fuchsian representations such that the set of isomorphism classes of marked Riemann surfaces $\{X_i\}_{i\in \NBbb}$,  $X_i=\Gamma_i\backslash \HBbb^2$, contains no limit points in $\Teich(\Sigma)$.  
We suppose furthermore that we have a constant $K$ and unique harmonic maps
$$u_i :\HBbb^2\lra T\ , \ v_i: \HBbb^2\to \HBbb^3$$
that are equivariant with respect to the action of $\pi_1(\Sigma)$, via $\sigma_i$ on the left, and $\rho_T$ and $\rho$ on the right, with $E(u_i)+E(v_i)\leq K$.

\smallskip\noindent{\bf Step 1.}
By a standard argument (see \cite{schoen-yau:incompressible, SacksUhlenbeck:82}), the energy bound plus Lemma \ref{lem:systole} imply that there is a uniform positive lower bound on the lengths of the shortest geodesics for the hyperbolic surfaces $X_i$. By the Mumford-Mahler compactness theorem, it follows that we can find quasiconformal homeomorphisms $g_i:\HBbb^2\to \HBbb^2$ and a Fuchsian representation $\sigma_\infty:\pi_1(S)\isorightarrow \Gamma_\infty$,  such that $g_i\circ\Gamma_i\circ g_i^{-1}=\Gamma_i$, and (after passing to a subsequence) $\hat\sigma_i=g_i\circ\sigma_i \circ g_i^{-1}\to \sigma_\infty$, in the Chabauty topology.   Introduce the following notation: for any $\gamma\in\pi_1(\Sigma)$, define
\begin{equation} \label{eqn:gamma-hat}
\hat\gamma_i:=\sigma_i^{-1}\circ\hat\sigma_i(\gamma)\ .
\end{equation}

\smallskip\noindent{\bf Step 2.}    
Let us first focus on the maps $u_i$ to the tree. By  \cite{KorevaarSchoen:1993} and the  convergence of the $\hat\sigma_i$, the 
maps $u_i$ are uniformly Lipschitz with a constant proportional to $\sqrt{E(u_i)}$.  In particular, since the energy is uniformly bounded, so is the Lipschitz constant.
Therefore, we may assume the Hopf differentials $\Phi_i$ of $u_i$, regarded as $\Gamma_i$-automorphic holomorphic quadratic differentials on $\HBbb^2$, converge $\Phi_i\to \Phi_\infty$ uniformly to a holomorphic differential $\Phi_\infty$. 
It is possible that $\Phi_\infty\equiv 0$; we will deal with this contingency in Step 6 below. In the intervening steps below, assume $\Phi_\infty \not\equiv 0$.

\smallskip\noindent{\bf Step 3.} As discussed previously, the leaf space $T_{\Phi_i}$ of the  vertical measured foliation of $\Phi_i$ has the structure of an $\RBbb$-tree  with an isometric action of $\pi_1(\Sigma)$ (via $\hat\sigma_i$) that is $\pi_1(\Sigma)$-equivariantly isometric to $T$. Denote this isometry by
$
\psi_i : T_{\Phi_i}\lra T
$. If we let $\pi_i : \HBbb^2\to T_{\Phi_i}$ be the projection onto the leaf space of the vertical foliation, 
then as in Section  \ref{sec:trees}  we have that  $u_i$ is given by $u_i = \psi_i\circ \pi_i$.

\smallskip\noindent{\bf Step 4.} Fix $\gamma\in\pi_1(\Sigma)$. We choose a representative curve $\alpha_\infty$ in $\HBbb^2$ from $0$ to $\sigma_\infty(\gamma)\cdot 0$ that is quasitransverse  to the vertical measured foliation of $\Phi_\infty$.  Let $\alpha_i:[0,1]\to \HBbb^2$  be a  path  
from $0$ to $\hat\sigma_i(\gamma)\cdot 0$, that is 
quasitransverse to the vertical foliation of
$\Phi_i$.  
 Then since the $\hat\sigma_i$ and $\Phi_i$ converge, $\alpha_i$  may furthermore be chosen $\varepsilon$-close 
 to  $\alpha_\infty$ 
 for $i$ sufficiently large.

\smallskip\noindent{\bf Step 5.}  By Step 4,  it follows that there is $I$  (depending on $\gamma$) such that  for $i$ sufficiently large,
$$
d_{T_{\Phi_i}}\left(\pi_i\alpha_i(1), \pi_i \alpha_i(0)\right)
 < I\ .
$$
On the other hand,
\begin{align*}
d_{T_{\Phi_i}}\left(\pi_i\alpha_i(1), \pi_i \alpha_i(0)\right)&=
d_{T}\left(\psi_i\circ\pi_i\alpha_i(1), \psi_i\circ\pi_i \alpha_i(0)\right) \\
&
=d_{T}\left(u_i(\hat\sigma_i(\gamma)\alpha_i(0)), u_i (\alpha_i(0))\right) \\
&
=d_{T}\left(u_i(\sigma_i(\hat\gamma_i)\alpha_i(0)), u_i (\alpha_i(0))\right) \\
&
=d_T(\rho_T(\hat\gamma_i)u_i(\alpha_i(0)), u_i(\alpha_i(0))    \ ,
\end{align*}
where  $\hat\gamma_i$ is defined by \eqref{eqn:gamma-hat}.
Hence, in particular, 
\begin{equation} \label{eqn:dt}
i(\hat\gamma_i, \lambda)=|\rho_T(\hat\gamma_i)|_T < I\ ,
\end{equation} 
 for $i$ sufficiently large.
 
 \smallskip\noindent{\bf Step 6.}  In the case where $\Phi_\infty\equiv 0$, it follows from \eqref{eqn:energy} that $E(u_i)\to 0$. Hence, by the assertion in Step 2, the Lipschitz constants for $u_i$ also tend to zero uniformly. Therefore, for any given $\gamma\in \pi_1(\Sigma)$, since $\hat\sigma_i(\gamma)\cdot 0\to \sigma_\infty(\gamma)\cdot 0$ remains bounded,
\begin{align*}
|\rho_T(\hat\gamma_i)|_T &\leq d_T(\rho_T(\hat\gamma_i) u_i(0), u_i(0)) \\
&=d_T(u_i(\sigma_i(\hat\gamma_i)\cdot 0), u_i(0)) \\
&=d_T(u_i(\hat\sigma_i(\gamma)\cdot 0), u_i(0)) \\
&\to 0%
\end{align*}
by the decay of the Lipschitz constants for $u_i$ and the convergence of $\hat\sigma_i(\gamma)\cdot 0$. Thus $|\rho_T(\hat\gamma_i)|_T <I$ for $i$ sufficiently large so that \eqref{eqn:dt}  holds in this case as well.

\smallskip\noindent{\bf Step 7.}  We apply a similar argument to the sequence of harmonic maps $v_i$.
Since the energy $E(v_i)$ is uniformly bounded,
and the groups  $g_i\circ\Gamma_i\circ g_i^{-1}$ converge, 
 the $v_i$ are uniformly Lipschitz.
  In particular,  for any $\gamma\in \pi_1(\Sigma)$ there is $B$ (depending on $\gamma$), such that
$$
d_{\HBbb^3}(v_i(\hat\sigma_i(\gamma)\cdot 0), v_i( 0)) \leq B\ .
$$
In that case,
\begin{align}
d_{\HBbb^3}(v_i(\hat\sigma_i(\gamma)\cdot 0), v_i( 0)) &=
d_{\HBbb^3}(v_i(\sigma_i(\hat\gamma_i)\cdot 0), v_i( 0))   \notag  \\
&=d_{\HBbb^3}(\rho(\hat\gamma_i) v_i( 0), v_i( 0))  \notag \\
\Longrightarrow \qquad \vert\rho(\hat\gamma_i)\vert_{\HBbb^3}&\leq B\ . \label{eqn:dh}
\end{align}
Of course, in this last term, the quantity $B$ still depends on $\gamma$ but is bounded independently of the index $i$.

\smallskip\noindent{\bf Step 8.}
We now relate the estimates of the previous three steps to arrive at  the following crucial conclusion. 
Combining eqs.\ \eqref{eqn:dt} and \eqref{eqn:dh} with Lemma \ref{lem: pleating locus}, we find that  the lengths $\ell_S(\hat\gamma_i)$ must be uniformly bounded in $i$. This implies that there are only finitely many homotopy classes among the $\hat\gamma_i$. Hence, after passing to a subsequence we may assume there exists a fixed $\hat\gamma$ such   that $\hat\gamma_i=\hat\gamma$, for  all  $i$.

\smallskip\noindent{\bf Step 9.}
Now apply the argument in Steps 4-8 to a set of generators $\gamma^{(1)}, \ldots, \gamma^{(2g)}$ of $\pi_1(\Sigma)$.  We conclude that along some subsequence, 
$$
\hat\gamma^{(j)}=\sigma_i^{-1}\circ\hat\sigma_i(\gamma^{(j)})\ , \ j=1, \ldots, 2g
$$
 (see \eqref{eqn:gamma-hat}).
But then the automorphisms $\sigma_i^{-1}\circ\hat\sigma_i$ are constant on all of $\pi_1(\Sigma)$. Since $\hat\sigma_i$ converges, so does $\sigma_i$, contradicting the hypothesis of no limit points for the $X_i$'s.

This   contradiction completes the proof.
\end{proof}
%

%
%

\nocite{} 
\bibliographystyle{alpha}
\bibliography{spectral}

\begin{thebibliography}{DDW00}

\bibitem[Bon86]{Bonahon:1986}
Francis Bonahon.
\newblock Bouts des vari\'et\'es hyperboliques de dimension {$3$}.
\newblock {\em Ann. of Math. (2)}, 124(1):71--158, 1986.

\bibitem[CM87]{CullerMorgan:1987}
Marc Culler and John~W. Morgan.
\newblock Group actions on {${\mathbb R}$}-trees.
\newblock {\em Proc. London Math. Soc. (3)}, 55(3):571--604, 1987.

\bibitem[Cor88]{Corlette:88}
Kevin Corlette.
\newblock Flat {$G$}-bundles with canonical metrics.
\newblock {\em J. Differential Geom.}, 28(3):361--382, 1988.

\bibitem[CS83]{CullerShalen:1983}
Marc Culler and Peter~B. Shalen.
\newblock Varieties of group representations and splittings of {$3$}-manifolds.
\newblock {\em Ann. of Math. (2)}, 117(1):109--146, 1983.

\bibitem[DDW00]{DaskalDostoglouWentworth:00}
G.~Daskalopoulos, S.~Dostoglou, and R.~Wentworth.
\newblock On the {M}organ-{S}halen compactification of the {${\rm
  SL}(2,{\mathbb C})$} character varieties of surface groups.
\newblock {\em Duke Math. J.}, 101(2):189--207, 2000.

\bibitem[Don87]{Donaldson:87}
S.~K. Donaldson.
\newblock Twisted harmonic maps and the self-duality equations.
\newblock {\em Proc. London Math. Soc. (3)}, 55(1):127--131, 1987.

\bibitem[GKM00]{gallo-kapovich-marden}
Daniel Gallo, Michael Kapovich, and Albert Marden.
\newblock The monodromy groups of {S}chwarzian equations on closed {R}iemann
  surfaces.
\newblock {\em Ann. of Math. (2)}, 151(2):625--704, 2000.

\bibitem[GW07]{GoldmanWentworth:07}
William~M. Goldman and Richard~A. Wentworth.
\newblock Energy of twisted harmonic maps of {R}iemann surfaces.
\newblock In {\em In the tradition of {A}hlfors-{B}ers. {IV}}, volume 432 of
  {\em Contemp. Math.}, pages 45--61. Amer. Math. Soc., Providence, RI, 2007.

\bibitem[Hit87a]{Hitchin:87}
Nigel~J. Hitchin.
\newblock The self-duality equations on a {R}iemann surface.
\newblock {\em Proc. London Math. Soc. (3)}, 55(1):59--126, 1987.

\bibitem[Hit87b]{Hitchin:1987b}
Nigel~J. Hitchin.
\newblock Stable bundles and integrable systems.
\newblock {\em Duke Math. J.}, 54(1):91--114, 1987.

\bibitem[HM79]{hubbard-masur:foliations}
John Hubbard and Howard Masur.
\newblock Quadratic differentials and foliations.
\newblock {\em Acta Math.}, 142(3-4):221--274, 1979.

\bibitem[Jos94]{Jost:1994}
J{\"u}rgen Jost.
\newblock Equilibrium maps between metric spaces.
\newblock {\em Calc. Var. Partial Differential Equations}, 2(2):173--204, 1994.

\bibitem[KP94]{kulkarni-pinkall:canonical}
Ravi~S. Kulkarni and Ulrich Pinkall.
\newblock A canonical metric for {M}\"obius structures and its applications.
\newblock {\em Math. Z.}, 216(1):89--129, 1994.

\bibitem[KS93]{KorevaarSchoen:1993}
Nicholas~J. Korevaar and Richard~M. Schoen.
\newblock Sobolev spaces and harmonic maps for metric space targets.
\newblock {\em Comm. Anal. Geom.}, 1(3-4):561--659, 1993.

\bibitem[KT92]{kamishima-tan:deformation-geometric}
Yoshinobu Kamishima and Ser~P. Tan.
\newblock Deformation spaces on geometric structures.
\newblock In {\em Aspects of low-dimensional manifolds}, volume~20 of {\em Adv.
  Stud. Pure Math.}, pages 263--299. Kinokuniya, Tokyo, 1992.

\bibitem[LM85]{LubotzkyMagid:1985}
Alexander Lubotzky and Andy~R. Magid.
\newblock Varieties of representations of finitely generated groups.
\newblock {\em Mem. Amer. Math. Soc.}, 58(336):xi+117, 1985.

\bibitem[PH92]{Penner-Harer}
Robert~C. Penner and John~L. Harer.
\newblock {\em Combinatorics of train tracks}, volume 125 of {\em Annals of
  Mathematics Studies}.
\newblock Princeton University Press, Princeton, NJ, 1992.

\bibitem[Sch12]{Schaposnik:Thesis}
Laura Schaposnik.
\newblock {\em Spectral Data for $G$-Higgs Bundles}.
\newblock Oxford University, 2012.
\newblock Ph.D. Thesis.

\bibitem[SU82]{SacksUhlenbeck:82}
J.~Sacks and K.~Uhlenbeck.
\newblock Minimal immersions of closed {R}iemann surfaces.
\newblock {\em Trans. Amer. Math. Soc.}, 271(2):639--652, 1982.

\bibitem[SY79]{schoen-yau:incompressible}
Richard~M. Schoen and Shing~Tung Yau.
\newblock Existence of incompressible minimal surfaces and the topology of
  three-dimensional manifolds with nonnegative scalar curvature.
\newblock {\em Ann. of Math. (2)}, 110(1):127--142, 1979.

\bibitem[Wen07]{Wentworth:07}
Richard~A. Wentworth.
\newblock Energy of harmonic maps and {G}ardiner's formula.
\newblock In {\em In the tradition of {A}hlfors-{B}ers. {IV}}, volume 432 of
  {\em Contemp. Math.}, pages 221--229. Amer. Math. Soc., Providence, RI, 2007.

\bibitem[Wol95]{wolf:trees}
Michael Wolf.
\newblock Harmonic maps from surfaces to {$\mathbb R$}-trees.
\newblock {\em Math. Z.}, 218(4):577--593, 1995.

\bibitem[Wol96]{wolf:hubbard-masur}
Michael Wolf.
\newblock On realizing measured foliations via quadratic differentials of
  harmonic maps to {$\mathbb R$}-trees.
\newblock {\em J. Anal. Math.}, 68:107--120, 1996.

\bibitem[Wol98]{wolf:mfs-harmonic}
Michael Wolf.
\newblock Measured foliations and harmonic maps of surfaces.
\newblock {\em J. Differential Geom.}, 49(3):437--467, 1998.

\end{thebibliography}

\end{document}